\documentclass[12pt,leqnod]{amsart}
\usepackage[utf8]{inputenc}
\usepackage[centertags]{amsmath}
\usepackage{newlfont}
\usepackage{amsmath}
\usepackage{amsthm}
\usepackage{amsfonts}
\usepackage{amssymb}
\usepackage{graphicx}
\usepackage{cancel}

\newtheorem{theorem}{Theorem}[section]
\newtheorem{definition}[theorem]{Definition}
\newtheorem{corollary}[theorem]{Corollary}
\newtheorem{lemma}[theorem]{Lemma}
\newtheorem{proposition}[theorem]{Proposition}
\newtheorem{remark}[theorem]{Remark}

\numberwithin{equation}{subsubsection}
\usepackage[top=2.5cm,bottom=2.5cm,left=2.5cm,right=2.4cm]{geometry}
\title[The Spectral Theorem for Quaternionic Normal Operators]{The Spectral Theorem for Quaternionic Normal Operators}
\author{El Hassan Benabdi and Mohamed Barraa}
\subjclass[2010]{47A10; 47A60; 46S10; 47S10}
\keywords{Quaternionic Hilbert space; Normal Operator; S-spectrum; spectral measure}
\address{Department of Mathematics, Faculty of Sciences Semlalia, Cadi Ayyad University, Marrakesh, Morocco} 
\email{elhassan.benabdi@gmail.com}
\email{barraa@hotmail.com}
\date{}
\begin{document}
\maketitle
\begin{abstract}
Let $\mathcal{H}$ be a right quaternionic Hilbert space and let $T$ be a bounded normal right quaternionic linear operator on $\mathcal{H}$. In this paper, we prove that there exists a unique spectral measure $E$ in $\mathcal{H}$ such that 
$$T=\int_{\sigma_S(T)}\lambda dE_\lambda,$$
where $\sigma_S(T)$ denotes the spherical spectrum of $T$.
\end{abstract}

\section{Introduction}
Let $\mathcal{H}$ be a complex Hilbert space and let $\mathcal{B}(\mathcal{H})$ denotes the set of all bounded quaternionic right linear operators from $\mathcal{H}$ to $\mathcal{H}$. Let $T\in\mathcal{B}(\mathcal{H})$ be normal, it is well-known that there exists a unique spectral measure $E$ in $\mathcal{H}$ such that 
\begin{align}
T=\int_{\sigma(T)}\lambda dE_\lambda,\label{6eq1}
\end{align}
where $\sigma(T)$ denotes the spectrum of $T$ (see, e.g., \cite[Theorem 3.15]{6C}). A large part of the theory of normal operators depends on this fact. \\

Several authors discussed the spectral theorem based on the S-spectrum \cite{6ACK,6CGK,6RS,6V}. Results related to the quaternionic spectral theorem can furthermore be found in \cite{6Gh,6RS2}.\\

The reader is encouraged to see the book of Colombo, Gantner and Kimsey \cite{6CGK} for a very good and detailed write-up of spectral theory based on the S-spectrum. In this book, the authors showed that if $T$ is a bounded normal right quaternionic linear operator on a right quaternionic Hilbert space $\mathcal{H}$, then there exists a unique spectral measure $E^i$ on $\Omega_i^+:=\sigma_S(T)\cap\mathbb{C}_i^+$ (see, \cite[Theorem 11.2.1]{6CGK}) such that
$$\langle  Tx,y\rangle=\int_{\Omega_i^+}\text{Re}(\lambda) d\langle E_\lambda^i x, y\rangle+\int_{\Omega_i^+}\text{Im}(\lambda) d\langle E_\lambda^i Jx, y\rangle\;\;\text{ for all }\;\; x,y\in\mathcal{H},$$ 
where $J$ is the imaginary component in the $T = A+JB$ decomposition (see, \cite[Theorem 9.3.5]{6CGK}) of the operator $T$.

In \cite{6Ghh}, the authors showed a version of the spectral theorem (see, \cite[Theorem 4.1]{6Ghh}). Then in \cite[Theorem 14.4.2]{6CGK}, the authors  rewrite \cite[Theorem 4.1]{6Ghh} in the terminology of spectral systems as follows.\\
Let $T = A + JB\in\mathcal{B}_R(\mathcal{H})$ be a normal operator. There exists a unique quaternionic spectral measure $E$ on the $\sigma$-algebra of axially symmetric Borel subsets of $\sigma_S(T)$ \big(the spherical spectrum of $T$\big), such that $(E, J)$ is a spectral system and such that
$$T=\int_{\sigma_S(T)}\lambda dE_J(\lambda).$$

In this paper, we give a canonical decomposition of a bounded normal right quaternionic linear operator on a right quaternionic Hilbert space using only the classical decomposition of bounded normal operators in the complex Hilbert spaces.\\

Before starting, we recall some basic definitions and results which are useful in the sequel.
\subsection*{Quaternions}
Let $\mathbb{H}$ be the set of all elements, called quaternions, of the form $q = a + b\mathrm{i} + c\mathrm{j} + d\mathrm{k}; a,b,c,d\in\mathbb{R}$, where $\mathrm{i},\mathrm{j}$ and $\mathrm{k}$ are quaternion units satisfying:
\begin{equation}
\label{6e1}
\mathrm{i}^2=\mathrm{j}^2=\mathrm{k}^2=\mathrm{i}\mathrm{j}\mathrm{k}=-1.
\end{equation}
Then $\mathbb{H}$ is a non-commutative $\mathbb{R}$-algebra with the addition defined same as in $\mathbb{C}$ and multiplication given by Equations (\ref{6e1}). For a given $q\in\mathbb{H}$, we define the real part of $q$, $\text{Re}(q):=a$ and the imaginary part of $q$, $\text{Im}(q):=b\mathrm{i} + c\mathrm{j} + d\mathrm{k}$. The conjugate and the modulus of $q$ are given respectively by
$$\bar{q}=a-b\mathrm{i}-c\mathrm{j}-d\mathrm{k},\;\vert q\vert=\sqrt{q\bar{q}}=\sqrt{a^2+b^2+c^2+d^2}.$$
The unit sphere of imaginary quaternions is given by
$$\mathbb{S}=\{q\in\mathbb{H}:q^2=-1\}.$$
Let $p$ and $q$ be two quaternions. $p$ and $q$  are said to be conjugated, if there is $s\in\mathbb{H}\setminus\{0\}$ such that $p=sqs^{-1}$. The set of all quaternions conjugated with $q$, is equal to the 2-sphere 
$$[q] =\{\text{Re}(q)+\vert\text{Im}(q)\vert i : i \in \mathbb{S}\}=\text{Re}(q)+\vert\text{Im}(q)\vert\mathbb{S}.$$
For every $U\subseteq\mathbb{H}$, $[U]$ denotes the set of all quaternions conjugated with the elements of $U$;
$$[U]:=\{[q]:q\in U\}.$$
For every $i \in \mathbb{S}$, denote by $\mathbb{C}_i$ the real subalgebra of $\mathbb{H}$ generated by $i$; that is, 
$$\mathbb{C}_i:=\{\alpha+\beta i\in\mathbb{H}: \alpha,\beta\in\mathbb{R}\}.$$
We say that $U\subseteq\mathbb{H}$ is axially symmetric if $[q] \subset U$ for every $q \in U$.\\

For a thorough treatment of the algebra of quaternions $\mathbb{H}$, the reader is referred, for instance, to \cite{6Gh}.
\subsection*{Quaternionc Hilbert spaces}
\begin{definition}
Let $(\mathcal{H},+)$ be an abelian group. $\mathcal{H}$ is a right quaternionic vector space if it is endowed with a right quaternionic multiplication $(\mathcal{H},\mathbb{H})\rightarrow \mathcal{H}$, $(x,q) \mapsto xq$ such that for all $x,y \in \mathcal{H}$ and all $p, q \in\mathbb{H}$, we have:
\begin{enumerate}
\item[(i)]  $(x+y)q = xq + yq$;
\item[(ii)]  $x(p + q) = xp + xq$;
\item[(iii)] $(xp)q = x(pq)$;
\item[(iv)]  $ x1=x$.
\end{enumerate}
\end{definition}
\begin{definition}
Let $\mathcal{H}$ be a right quaternionic vector space. A Hermitian quaternionic scalar product is a map $\langle\cdot,\cdot\rangle:\mathcal{H}\times \mathcal{H}\rightarrow \mathbb{H},(x,y)\mapsto \langle x,y\rangle$ that satisfies the following properties.
\begin{itemize}
\item[(i)] $\langle xp+yq,z\rangle=\langle x,z\rangle p+\langle y,z\rangle q$;
\item[(ii)] $\overline{\langle x,y\rangle}=\langle y,x\rangle$;
\item[(iii)] $\langle x,x\rangle\geq0$ and $\langle x,x\rangle=0\Leftrightarrow x=0$;
\end{itemize}
for all $p,q\in\mathbb{H}$ and all $x, y, z \in \mathcal{H}$.
\end{definition}
Suppose that $\mathcal{H}$ is equipped with such a Hermitean quaternionic scalar product. Then we can define the quaternionic norm $\Vert\cdot\Vert : \mathcal{H}\rightarrow \mathbb{R}^+$ by setting
$$\Vert x\Vert=\sqrt{\langle x,x\rangle}\;\;\;\text{ for every }\;\;\; x\in\mathcal{H}.$$

A right quaternionic vector space $\mathcal{H}$ is said to be a right quaternionic Hilbert space if it is equipped with a Hermitean quaternionic scalar product which is complete with respect to its natural distance $d(x,y):=\Vert x-y\Vert$ for all $x,y\in\mathcal{H}$. 
\begin{definition}
Let $\mathcal{H}$ be a right quaternionic Hilbert space. A right linear operator on $\mathcal{H}$ is a map $T : \mathcal{H} \rightarrow \mathcal{H}$ such that:
$$ T(xq+y) = (Tx)q + Ty\;\;\text{ for all }\;\; x, y\in \mathcal{H}\;\;\text{ and all }\;\;q\in\mathbb{H}.$$
A right linear operator $T$ on $\mathcal{H}$ is called bounded if
$$\Vert T\Vert:=\sup\{\Vert Tx\Vert: x\in \mathcal{H}, \Vert x\Vert=1\}<\infty.$$
For all $r\in\mathbb{R}$ and all $x\in \mathcal{H}$. Define  $(Tr)x := (Tx)r$ and $rT := Tr$. The adjoint of $T$ will be denoted by $T^*$. The set of all right linear bounded operators on $\mathcal{H}$ is denoted by $\mathcal{B}_R(\mathcal{H})$. For $T\in\mathcal{B}_R(\mathcal{H})$, let $\mathcal{N}(T)$ denote the null space of $T$, and let $\mathcal{R}(T)$ denote the range of $T$.
\end{definition}
\subsection*{Spherical spectrum}
Colombo, Sabadini and Struppa \cite{6CSS} extended the definitions of the spectrum and resolvent in quaternionic Hilbert spaces as follows.
\begin{definition}
Let $\mathcal{H}$ be a right quaternionic Hilbert space and let $T \in \mathcal{B}_R(\mathcal{H})$. For $q \in\mathbb{H}$, we set
$$Q_q(T) := T^2 -2\text{Re}(q)T + \vert q\vert^2I.$$
Where $I$ is the identity operator on $\mathcal{H}$.
We define the spherical resolvent set $\rho_S(T)$ of $T$ as
$$\rho_S(T) := \{q \in\mathbb{H} : Q_q(T) \text{ is invertible in } \mathcal{B}_R(\mathcal{H})\},$$
and we define the spherical spectrum $\sigma_S(T)$ of $T$ as  
$$\sigma_S(T):=\mathbb{H}\setminus\rho_S(T).$$
\end{definition}
\begin{proposition}
\label{6p2}
Let $T\in\mathcal{B}_R(\mathcal{H})$. Then $\sigma_S(T)$ is a nonempty compact and axially symmetric set.
\end{proposition}

For more details on the notion of spherical spectrum, the reader may consult, e.g., \cite{6BB,6CGK,6CSS,6Gh}.
\section{The Spectral Theorem and Consequences}
In the following, $\mathcal{H}$ will be a fixed right quaternionic Hilbert space and $\langle\cdot,\cdot\rangle$ denotes the inner product in $\mathcal{H}$.
\begin{lemma}
\label{6l1}
Let $i\in\mathbb{S}$. The application $\langle\cdot,\cdot\rangle_i:\mathcal{H}\times \mathcal{H}\rightarrow \mathbb{C}_i,(x,y)\mapsto\mathrm{Re}\big(\langle x,y\rangle\big)-\mathrm{Re}\big(\langle x,y\rangle i\big)i$ satisfies the following properties.
\begin{itemize}
\item[(i)] $\langle xp+yq,z\rangle_i=\langle x,z\rangle_ip+\langle y,z\rangle_iq$;
\item[(ii)] $\overline{\langle x,y\rangle_i}=\langle y,x\rangle_i$;
\item[(iii)] $\langle x,x\rangle_i\geq0$ and $\langle x,x\rangle_i=0\Leftrightarrow x=0$;
\end{itemize}
for all $p,q\in\mathbb{C}_i$ and all $x, y, z \in \mathcal{H}$.
\end{lemma}
\begin{proof}
Let $p=\alpha+\beta i\in\mathbb{C}_i$ with $\alpha,\beta\in\mathbb{R}$ and let $x, y\in \mathcal{H}$. Then\\

\noindent (i) $\langle xp,y\rangle_i=\langle x(\alpha+\beta i),y\rangle_i=\mathrm{Re}\big(\langle x(\alpha+\beta i),y\rangle\big)-\mathrm{Re}\big(\langle x(\alpha+\beta i),y\rangle i\big)i=\mathrm{Re}\big(\langle x,y\rangle\big)\alpha+\mathrm{Re}\big(\langle x,y\rangle i\big)\beta-\mathrm{Re}\big(\langle x,y\rangle i\big)i\alpha+\mathrm{Re}\big(\langle x,y )\rangle\big)i\beta=\langle x,y\rangle_ip.$\\

\noindent (ii) $\langle y,x\rangle_i=\mathrm{Re}\big(\langle y,x\rangle\big)-\mathrm{Re}\big(\langle y,x\rangle i\big)i=\mathrm{Re}\big(\overline{\langle x,y\rangle}\big)+\mathrm{Re}\big(\overline{i\langle x,y\rangle}\big)i=\mathrm{Re}\big(\langle x,y\rangle\big)+\mathrm{Re}\big(i\langle x,y\rangle\big)i=\mathrm{Re}\big(\langle x,y\rangle\big)+\mathrm{Re}\big(\langle x,y\rangle i\big)i=\overline{\langle x,y\rangle_i}.$\\

\noindent (iii) $\langle x,x\rangle_i=\langle x,x\rangle\geq0$ and $\langle x,x\rangle_i=0\Leftrightarrow \langle x,x\rangle=0 \Leftrightarrow x=0$.
\end{proof}
\begin{definition}
Let $i\in\mathbb{S}$. We denote the space $\mathcal{H}$ considered as a complex vector space over the complex field $\mathbb{C}_i$ by $\mathcal{H}_i$ with respect to the structure induced by the right quaternionic Hilbert space $\mathcal{H}$; its sum is the sum of $\mathcal{H}$, its complex scalar multiplication is the right scalar multiplication of $\mathcal{H}$ restricted to $\mathbb{C}_i$.
\end{definition}
\begin{corollary}
\label{6c1}
For any $i\in\mathbb{S}$, the map $\langle\cdot,\cdot\rangle_i$ is an inner product on $\mathcal{H}_i$. Moreover, $\mathcal{H}_i$ is a complex Hilbert space with respect to the norm $\Vert\cdot\Vert_i:=\sqrt{\langle\cdot,\cdot\rangle_i}$.
\end{corollary}
\begin{remark}
Let $x,y\in \mathcal{H}$ and let $i\in\mathbb{S}$. Then 
\begin{itemize}
\item[(i)] If $\langle x,y\rangle\in\mathbb{C}_i$. Then $\langle x,y\rangle_i=\langle x,y\rangle$.
\item[(ii)] $\langle x,x\rangle_i=\langle x,x\rangle$, hence $\Vert \cdot\Vert_i=\Vert \cdot\Vert$ where $\Vert \cdot\Vert$ is the norm on $\mathcal{H}$ and then the topologies of $\mathcal{H}$ and $\mathcal{H}_i$ coincide.
\end{itemize}
\end{remark}
\begin{definition}
\label{6d1}
Let $\Omega$ be a nonempty axially symmetric Borel subset of $\mathbb{H}$ and let $\mathcal{A}_{\Omega}^s$ be the $\sigma$-algebra of axially symmetric Borel subsets of $\Omega$. A spectral measure in a right quaternionic Hilbert space $\mathcal{H}$ is a mapping $E:\mathcal{A}_{\Omega}^s\rightarrow \mathcal{B}_R(\mathcal{H})$ such that
\begin{itemize}
\item[(i)] $E(\omega)$ is a bounded orthogonal projection for every $\omega\in\mathcal{A}_{\Omega}^s$;
\item[(ii)] $E(\emptyset) = 0$ and $E(\Omega) = I$;
\item[(iii)] $E(\omega_1\cap\omega_2) = E(\omega_1)E(\omega_2)$ for every $\omega_1,\omega_2\in\mathcal{A}_{\Omega}^s$;
\item[(iv)] $E(\cup_{n}\omega_n) =\sum_{n} E(\omega_n)$ whenever $\{\omega_n\}$ is a countable collection of pairwise disjoint sets in $\mathcal{A}_{\Omega}^s$ (i.e., $E$ is countably additive).
\end{itemize} 
\end{definition}

Let $T\in\mathcal{B}_R(\mathcal{H})$ and $i\in\mathbb{S}$. Let $T_i$ be the linear operator on $\mathcal{H}_i$ defined by $T_i(x):=Tx$ for all $x\in \mathcal{H}_i$. Clearly, $T_i\in\mathcal{B}(\mathcal{H}_i)$. In order to show the spectral theorem, we need the following lemma.
\begin{lemma}
\label{6l4}
Let $T\in\mathcal{B}_R(\mathcal{H})$ be normal and let $i\in\mathbb{S}$. Then $(T_i)^*=(T^*)_i$, and furthermore, if $T$ is normal, then so is $T_i$.
\end{lemma}
\begin{proof}
Let $x,y\in \mathcal{H}_i$. We have 
\begin{align*}
\langle T_ix,y\rangle_i&=\mathrm{Re}\big(\langle T_ix,y\rangle\big)-\mathrm{Re}\big(\langle T_ix,y\rangle i\big)i=\mathrm{Re}\big(\langle Tx,y\rangle\big)-\mathrm{Re}\big(\langle Tx,y\rangle i\big)i\\
&=\mathrm{Re}\big(\langle x,T^*y\rangle\big)-\mathrm{Re}\big(\langle x,T^*y\rangle i\big)i=\langle x,T^*y\rangle_i=\langle x,(T^*)_iy\rangle_i.
\end{align*}
Therefore, $(T_i)^*=(T^*)_i$. Moreover, if $T$ is normal, then $(T_i)^*T_i=(T^*)_iT_i=(T^*T)_i=(TT^*)_i=T_i(T^*)_i=T_i(T_i)^*$. Hence $T_i$ is normal.
\end{proof}
Let $T\in\mathcal{B}_R(\mathcal{H})$ be normal and let $i\in\mathbb{S}$. By \cite[Theorem 3.15]{6C}, there exists a unique complex spectral measure $E^i$ of the spectral decomposition of $T_i$ \big($E^i:\mathcal{A}_{\sigma(T_i)}\rightarrow \mathcal{B}(\mathcal{H}_i)$ where $\mathcal{A}_{\sigma(T_i)}$ is the $\sigma$-algebra of Borel subsets of $\sigma(T_i)$\big) such that
$$T_i=\int_{\sigma(T_i)}\lambda dE^i_\lambda.$$
That is $\langle  T_ix,y\rangle_i=\int_{\sigma(T_i)}\lambda d\langle E^i_\lambda x, y\rangle_i\;$ for all $x,y\in \mathcal{H}_i$.\\

Let $\omega\in\mathcal{A}_{\sigma(T_i)}$, then $E^i(\omega)$ is linear on $\mathcal{H}_i$. Let $\mathcal{A}_{\sigma(T_i)}^s$ be  the $\sigma$-algebra of Borel symmetric (with respect to the real axis) subsets  of $\sigma(T_i)$, the following lemma tells us that $E^i(\omega)$ is actually right quaternionic linear on $\mathcal{H}$ for all $\omega\in\mathcal{A}_{\sigma(T_i)}^s$.
\begin{lemma}
\label{6l3}
Let $T\in\mathcal{B}_R(\mathcal{H})$ be normal and $i\in\mathbb{S}$. Let $E^i$ be the complex spectral measure of the spectral decomposition of $T_i$. Then
$$E^i(\omega)\in\mathcal{B}_R(\mathcal{H}) \;\;\text{ for all }\;\; \omega\in\mathcal{A}_{\sigma(T_i)}^s.$$
\end{lemma}
\begin{proof}
An analytic extension of $z\mapsto (I_iz-T_i)^{-1}x$ with $x\in\mathcal{H}_i$ is a holomorphic function $f$ defined on a set $\mathcal{D}(f)\supset\rho(T_i)$ \big(where $\rho(T_i)$ is the resolvent set of $T_i$\big) such that $(I_iz-T_i)f(z)=x$ for all $z\in \mathcal{D}(f)$. $T_i$ is a bounded spectral operator in $\mathcal{H}_i$, then the function $z\mapsto (I_iz-T_i)^{-1}x$ has a unique maximal analytic extension. The local resolvent $\rho(x)$ is the domain of the unique maximal analytic extension of $z\mapsto(I_iz-T_i)^{-1}x$, and the local spectrum $\sigma(x)$ is the complement of $\rho(x)$ in $\mathbb{C}_i$ \big(see, \cite[Chapter XV]{6Dun}\big).\\
Let $\omega\in\mathcal{A}_{\sigma(T_i)}^s$ and let $\overline{\omega}$ be the closure of $\omega$. By \cite[ Theorem XV.3.4]{6Dun}, we have
$$\mathcal{R}\big(E^i(\overline{\omega})\big)=\{x\in\mathcal{H}_i:\sigma(x)\subseteq\overline{\omega}\}.$$
Let $j\in\mathbb{S}$ such that $ij=-ji$, then $\{1,i,j,ij\}$ is a basis of $\mathbb{H}$. It is sufficient to show that $E^i(\omega)(xj)=E^i(\omega)(x)j$ for all $x\in\mathcal{H}$.\\
Let $f$ be the unique maximal analytic extension of $z\mapsto(I_iz-T_i)^{-1}x$ defined on $\rho(x)$. The mapping $g:z\mapsto f(\bar{z})j$ defined on $\rho(x)^*:=\{\bar{z}:z\in\rho(x)\}$ is holomorphic. Indeed,
$$\lim_{h\rightarrow 0}\big(f(\overline{z+h})j-f(\bar{z})j\big)\frac{1}{h}=\lim_{h\rightarrow 0}\big(f(\bar{z}+\bar{h})-f(\bar{z})\big)\frac{1}{\bar{h}}j=f^\prime(\bar{z})j.$$
Moreover, if $z\in\rho(x)^*$, then
\begin{align*}
(I_iz-T_i)g(z)&=(I_iz-T_i)(f(\bar{z})j)=f(\bar{z})jz-T(f(\bar{z})j)=f(\bar{z})\bar{z}j-T(f(\bar{z}))j\\
&=\big((I_i\bar{z}-T_i)f(\bar{z})\big)j=xj.
\end{align*}
Hence $z\mapsto f(\bar{z})j$ is an analytic extension of $z\mapsto(I_iz-T_i)^{-1}(xj)$ that is defined on $\rho(x)^*$. Consequently $\rho(x)^*\subseteq\rho(xj)$. Similar arguments show that $\rho(xj)^*\subseteq\rho(-x)=\rho(x)$. Altogether, we obtain $\rho(x)=\rho(xj)^*$, and so $\sigma(xj)=\sigma(x)^*$. Thus
\begin{align}
\mathcal{R}\big(E^i(\overline{\omega})\big)=\{x\in\mathcal{H}_i:\sigma(x)\subseteq\overline{\omega}\}=\{xj\in\mathcal{H}_i:\sigma(x)\subseteq\overline{\omega}\}=\mathcal{R}\big(E^i(\overline{\omega})\big)j.
\label{6eq3}
\end{align}
If we choose an increasing sequence of closed sets $\omega_n\in\mathcal{A}_{\sigma(T_i)}^s$ with $\cup \omega_n=\sigma(T_i)\setminus\overline{\omega}$, we therefore have $E^i\big(\sigma(T_i)\setminus\overline{\omega}\big)x=\lim_{n\rightarrow\infty}E^i(\omega_n)x$ for all $x\in\mathcal{H}$. Hence by (\ref{6eq3}), it follows that
\begin{align}
\mathcal{R}\big(E^i(\sigma(T_i)\setminus\overline{\omega})\big)=\mathcal{R}\big(E^i(\sigma(T_i)\setminus\overline{\omega})\big)j.
\label{6eq4}
\end{align}
Now observe that 
$$\mathcal{A}:=\{\Lambda\in\mathcal{A}_{\sigma(T_i)}^s:\mathcal{R}\big(E^i(\Lambda)\big)=\mathcal{R}\big(E^i(\Lambda)\big)j\;\text{ and }\;\mathcal{R}\big(E^i(\sigma(T_i)\setminus\Lambda)\big)=\mathcal{R}\big(E^i(\sigma(T_i)\setminus\Lambda)\big)j\}$$
is a $\sigma$-algebra that contains closed sets. Hence $\mathcal{A}=\mathcal{A}_{\sigma(T_i)}^s.$ Therefore,
\begin{align}
\mathcal{R}\big(E^i(\omega)\big)=\mathcal{R}\big(E^i(\omega)\big)j,
\end{align} 
and so
$$\mathcal{N}\big(E^i(\omega)\big)=\mathcal{R}\big(E^i(\sigma(T_i)\setminus\omega)\big)=\mathcal{R}\big(E^i(\sigma(T_i)\setminus\omega)\big)j=\mathcal{N}\big(E^i(\omega)\big)j.$$
Since $\mathcal{H}=\mathcal{R}\big(E^i(\omega)\big)\oplus\mathcal{N}\big(E^i(\omega)\big)$, $xj=E^i(\omega)(xj)+\big(I-E^i(\omega)\big)(xj)$ and  $xj=E^i(\omega)(x)j+\big(I-E^i(\omega)\big)(x)j$ for all $x\in\mathcal{H}$. It follows that $E^i(\omega)(xj)=E^i(\omega)(x)j$. This completes the proof.
\end{proof}
Let $T\in\mathcal{B}_R(\mathcal{H})$ be normal and let $i,j\in\mathbb{S}$. The following lemma gives a relation between the complex spectral measures $E^i$ and $E^j$ of the spectral decompositions of $T_i$ and $T_j$, respectively.
\begin{lemma}
\label{6l2}
Let $i,j\in\mathbb{S}$ and $s\in\mathbb{H}\setminus\{0\}$ such that $s^{-1}i s=j$. Let $T\in\mathcal{B}_R(\mathcal{H})$ be normal, and let $E^i$ and $E^j$ be the complex spectral measures of the spectral decompositions of $T_i$ and $T_j$, respectively. Then
$$E^i(\omega)=E^j\big(s^{-1}\omega s\big) \;\;\text{ for all }\;\; \omega\in\mathcal{A}_{\sigma(T_i)}^s,$$
where $s^{-1}\omega s:=\{s^{-1}q s:q\in\omega\}$.
\end{lemma}
\begin{proof}
Let $x\in \mathcal{H}$ and let $P\in\mathbb{R}[X,Y]$. Define 
\begin{align}
p(T):=P(T,T^*)\;\;\text{ and }\;\;  p(\lambda):=P(\lambda,\bar{\lambda})\;\;\text{ for all }\lambda\in\mathbb{H}.\label{6eq2}
\end{align}
We have  
$$\langle p(T_j)^*p(T_j) x,x\rangle_j=\int_{\sigma(T_j)}\vert p(\lambda)\vert^2 d\langle E^j_\lambda x, x\rangle_j=\int_{\sigma(T_j)}\vert p(\lambda)\vert^2 d\langle E^j_\lambda x, x\rangle.$$
By \cite[Proposition 3.8]{6BB}, it follows that $\sigma(T_j)=s^{-1}\sigma(T_i)s$, hence
$$\int_{\sigma(T_j)}\vert p(\lambda)\vert^2 d\langle E^j_\lambda x, x\rangle=\int_{\sigma(T_i)}\vert p(\lambda)\vert^2 d\langle E^j_{s^{-1}\lambda s} x, x\rangle.$$
On the other hand, we have
$$\langle p(T_i)^*p(T_i) x,x\rangle_i=\int_{\sigma(T_i)}\vert p(\lambda)\vert^2 d\langle E^i_\lambda x, x\rangle_i=\int_{\sigma(T_i)}\vert p(\lambda)\vert^2 d\langle E^i_\lambda x, x\rangle.$$
Since $\langle p(T)^*p(T) x,x\rangle\in\mathbb{R}$, $\langle p(T_i)^*p(T_i) x,x\rangle_i=\langle p(T_j)^*p(T_j) x,x\rangle_j$. Thus
$$\int_{\sigma(T_i)}\vert p(\lambda)\vert^2 d\langle E^j_{s^{-1}\lambda s} x, x\rangle=\int_{\sigma(T_i)}\vert p(\lambda)\vert^2 d\langle E^i_\lambda x, x\rangle.$$
Set 
$$\mu_x(\omega):=\langle E^i(\omega) x, x\rangle\;\;\text{ and }\;\; \nu_x(\omega):=\langle E^j(s^{-1}\omega s) x, x\rangle\;\; \text{ for all }\;\;\omega\in\mathcal{A}_{\sigma(T_i)}.$$
Let $\mathcal{P}$ be the set of all polynomials $P$ in the real variables $\alpha,\beta$ and with real coefficients $\sigma(T_i):\;\lambda=\alpha+\beta i\mapsto P(\alpha,\beta)$. By the Stone-Weierstrass Theorem for real functions, $\mathcal{P}$ is dense in the Banach space $\big(\mathcal{C}(\sigma(T_i)),\Vert\cdot\Vert_\infty\big)$ \big(the set of all real-valued continuous functions on $\sigma(T_i)$\big). Since $\sigma(T_i)$ is a compact set in $\mathbb{C}_i$ and $\mu_x$ is a finite positive measure on $\mathcal{A}_{\sigma(T_i)}$, it follows that $\mathcal{P}$ is also dense in the normed space $\big(\mathcal{C}(\sigma(T_i)),\Vert\cdot\Vert_2\big)$, which in turn is a dense linear manifold of the Banach space $\big(L^2(\sigma(T_i),\mu_x), \Vert\cdot\Vert_2\big)$, and hence $\mathcal{P}$ is dense in $\left(L^2\big(\sigma(T_i),\mu_x\big), \Vert\cdot\Vert_2\right)$. The same conclusion holds for $\nu_x$. It follows that
$$\int_{\sigma(T_i)}\chi_\omega(\lambda) d\langle E^j_{s^{-1}\lambda s} x, x\rangle=\int_{\sigma(T_i)}\chi_\omega(\lambda) d\langle E^i_\lambda x, x\rangle\;\;\text{ for all }\;\; \omega\in\mathcal{A}_{\sigma(T_i)}.$$
That is $\langle E^j(s^{-1}\omega s) x, x\rangle=\langle E^i(\omega) x, x\rangle$ for all $\omega\in\mathcal{A}_{\sigma(T_i)}$.\\ Let $\omega\in\mathcal{A}_{\sigma(T_i)}^s$. By Lemma \ref{6l3}, $E^i(\omega)\in\mathcal{B}_R(\mathcal{H})$ and so $E^i(\omega)$ is $\mathbb{C}_j$-linear. Thus $\langle \big(E^i(\omega)-E^j(s^{-1}\omega s)\big) x, x\rangle_j=\langle \big(E^i(\omega)-E^j(s^{-1}\omega s)\big) x, x\rangle=0$ for all $x\in\mathcal{H}_j$, hence $E^i(\omega)-E^j(s^{-1}\omega s)=0$ (see, e.g., \cite[Theorem 2.14]{6C}). It follows that $E^i(\omega)=E^j(s^{-1}\omega s)$. This completes the proof.
\end{proof}
We are now ready to show the main result of this paper: the canonical decomposition of a bounded normal right quaternionic linear operator, the quaternionic analogue of \cite[Theorem 3.15]{6C}.
\begin{theorem}
\label{6t1}
Let $T\in\mathcal{B}_R(\mathcal{H})$ be normal and let $J$ be the imaginary component in the $T = A+JB$ decomposition of $T$. There exists a unique spectral measure $E$ in $\mathcal{H}$ such that 
$$\langle  Tx,y\rangle=\int_{\sigma_S(T)}\text{Re}(\lambda) d\langle E_\lambda x, y\rangle+\int_{\sigma_S(T)}\vert\text{Im}(\lambda)\vert d\langle E_\lambda Jx, y\rangle\;\text{ for all }\; x,y\in\mathcal{H}.$$
Notation: $$T=\int_{\sigma_S(T)}\lambda dE_\lambda.$$
\end{theorem}
\begin{proof}
We shall split the proof into three parts.\\

\noindent (i) Let $i\in\mathbb{S}$ and let $E^i$ be the complex spectral measure of the spectral decomposition of $T_i$. We claim that the map $E:\mathcal{A}_{\sigma_S(T)}^s \rightarrow\mathcal{B}_R(\mathcal{H})$ defined by
$E\big(\Omega\big)=E^i\big(\Omega\cap\mathbb{C}_i\big)$ is a spectral measure in $\mathcal{H}$.\\
To prove this, first note that $\mathcal{A}_{\sigma(T_i)}^s=\mathcal{A}_{\sigma_S(T)}^s\cap\mathbb{C}_i:=\{\Omega\cap\mathbb{C}_i:\Omega\in\mathcal{A}_{\sigma_S(T)}^s\}$. Indeed, $\sigma(T_i)=\sigma_S(T)\cap\mathbb{C}_i$ \big(see, \cite[Proposition 3.8]{6BB}\big) and $\mathbb{C}_i$ is closed in $\mathbb{H}$. Then one can easily check that $E$ is a spectral measure in $\mathcal{H}$.\\
Now, let us prove that $E$ is independent of the choice of the imaginary unit $i\in\mathbb{S}$. Let $j\in\mathbb{S}$, $s\in\mathbb{H}\setminus\{0\}$ such that $s^{-1}i s=j$ and $E^j$ be the complex spectral measure of the spectral decomposition of $T_j$. We have  
$$\Omega\cap\mathbb{C}_j=s^{-1}\big(\Omega\cap\mathbb{C}_i\big)s\;\;\text{ for all }\;\;\Omega\in\mathcal{A}_{\sigma_S(T)}^s.$$
Hence by Lemma \ref{6l2}, 
$$E\big(\Omega\big)=E^i\big(\Omega\cap\mathbb{C}_i\big)=E^j\big(s^{-1}\big(\Omega\cap\mathbb{C}_i\big)s\big)=E^j\big(\Omega\cap\mathbb{C}_j\big).$$

\noindent (ii) Let $A$, $B$, and $J$ be as in \cite[Theorem 9.3.5]{6CGK}. Let $x,y\in \mathcal{H}$ and let $i\in\mathbb{S}$ such that $\langle T x,y\rangle\in\mathbb{C}_i$. Then 
\begin{align*}
\langle  A_ix,y\rangle_i=\int_{\sigma(T_i)}\text{Re}(\lambda) d\langle E^i_\lambda x, y\rangle_i\;\text{ and }\;
\langle  B_iJx,y\rangle_i=\int_{\sigma(T_i)}\vert\text{Im}(\lambda)\vert d\langle E^i_\lambda Jx, y\rangle_i.
\end{align*}
It follows that
\begin{align*}
\langle  Tx,y\rangle=\langle  (A+BJ)x,y\rangle_i=\int_{\sigma(T_i)}\text{Re}(\lambda) d\langle E^i_\lambda x, y\rangle_i+\int_{\sigma(T_i)}\vert\text{Im}(\lambda)\vert d\langle E^i_\lambda Jx, y\rangle_i.
\end{align*}
Let $\omega\in\mathcal{A}_{\sigma(T_i)}$, by the polarization identity, we have  
\begin{align*}
\langle E^i(\omega) x, y\rangle_i=&\frac{1}{4}\Big(\langle E^i(\omega) (x+y), x+y\rangle_i-\langle E^i(\omega) (x-y), x-y\rangle_i\\
&+i\langle E^i(\omega) (x+yi), x+yi\rangle_i-i\langle E^i(\omega) (x-yi), x-yi\rangle_i\Big).
\end{align*}
Since $\langle E^i(\omega)z, z\rangle_i\in\mathbb{R}$ for all $z\in\mathbb{H}$, $\langle E^i(\omega) x, y\rangle_i\in\mathbb{C}_i$ and so $\langle E^i(\omega) x, y\rangle_i=\langle E^i(\omega) x, y\rangle$. Thus 
$$\langle  Tx,y\rangle=\int_{\sigma(T_i)}\text{Re}(\lambda) d\langle E^i_\lambda x, y\rangle+\int_{\sigma(T_i)}\vert\text{Im}(\lambda)\vert d\langle E^i_\lambda Jx, y\rangle.$$
Let $E$ be the spectral measure constructed in part (i), it is easy to see that $E^i\big(\omega\big)=E\big([\omega]\big)$ for all $\omega\in\mathcal{A}_{\sigma(T_i)}^s$. Thus for any continuous real slice function $f$ on $\sigma_S(T)$, we have
$$\int_{\sigma(T_i)}f(\lambda) d\langle E^i_\lambda u, v\rangle=\int_{\sigma_S(T)}f(\lambda) d\langle E_\lambda u, v\rangle\;\text{ for all }\;u,v\in\mathcal{H}.$$
Therefore, 
$$\langle  Tx,y\rangle=\int_{\sigma_S(T)}\text{Re}(\lambda) d\langle E_\lambda x, y\rangle+\int_{\sigma_S(T)}\vert\text{Im}(\lambda)\vert d\langle E_\lambda Jx, y\rangle.$$

\noindent (iii) Uniqueness of the spectral measure $E :\mathcal{A}_{\sigma_S(T)}^s\rightarrow\mathcal{B}_R(\mathcal{H})$ is proved as follows. For all $P\in\mathbb{R}[X,Y]$, let $p(T)$ be as in (\ref{6eq2}), then
$$\int_{\sigma(T_i)}\vert p(\lambda)\vert^2 d\langle E^i_\lambda x, x\rangle_i=\langle p(T)^*p(T) x,x\rangle=\int_{\sigma_S(T)}\vert p(\lambda)\vert^2 d\langle E_\lambda x, x\rangle.$$ 
Thus (see proof of Lemma \ref{6l2})
$$\int_{\sigma(T_i)}\chi_{\Omega}(\lambda) d\langle E^i_{\lambda } x, x\rangle=\int_{\sigma_S(T)}\chi_\Omega(\lambda) d\langle E_\lambda x, x\rangle\;\;\text{ for all }\;\; \Omega\in\mathcal{A}_{\sigma_S(T)}^s.$$
And so  $E^i(\Omega\cap\mathbb{C}_i)=E(\Omega)$ for all $\Omega\in\mathcal{A}_{\sigma_S(T)}^s$. Since $E^i$ is the spectral measure of the spectral decomposition of $T_i$, $E^i$ is unique then so is $E$.
\end{proof}
\begin{lemma}
Let $T=\int_{\sigma_S(T)}\lambda dE_\lambda$ be the spectral decomposition of a normal operator $T\in\mathcal{B}_R(\mathcal{H})$ as in Theorem $\ref{6t1}$. If $\Omega\in\mathcal{A}_{\sigma_S(T)}^s$, then
\begin{itemize}
\item[(i)] $E(\Omega)\neq 0$ whenever $\Omega$ is open in the relative topology of $\sigma_S(T)$.
\item[(ii)] $E(\Omega)T = TE(\Omega)$. 
\item[(iii)] $\sigma_S(T\vert_{\mathcal{R}(E(\Omega))})\subseteq\overline{\Omega}$.
\end{itemize} 
\end{lemma}
\begin{proof}
Let $i\in\mathbb{S}$ and let $E^i$ be the spectral measure of the spectral decomposition of $T_i$.\\ 

\noindent (i) We have $E\big(\Omega\big)=E^i\big(\Omega\cap\mathbb{C}_i\big)$. Since $\Omega$ is open in the relative topology of $\sigma_S(T)$, $\Omega\cap\mathbb{C}_i$ is open relative to the topology of $\sigma(T_i)$. Then by \cite[Theorem 3.15]{6C}, $E^i\big(\Omega\cap\mathbb{C}_i\big)\neq0$, and so $E\big(\Omega\big)\neq0$.\\

\noindent (ii) By \cite[Lemma 3.16]{6C} it follows that $E^i(\Omega\cap\mathbb{C}_i)T_i = T_iE^i(\Omega\cap\mathbb{C}_i)$. Then clearly $E(\Omega)T = TE(\Omega)$.\\

\noindent (iii) It follows from \cite[Proposition 3.8]{6BB},  that $\sigma_S(T\vert_{\mathcal{R}(E(\Omega))})=[\sigma(T_i\vert_{\mathcal{R}(E(\Omega))_i})]$ and by \cite[Lemma 3.16]{6C}, we get $\sigma(T_i\vert_{\mathcal{R}(E(\Omega))_i})=\sigma(T_i\vert_{\mathcal{R}(E^i(\Omega\cap\mathbb{C}_i))_i})\subseteq\overline{\Omega\cap\mathbb{C}_i}$. Thus $[\sigma(T_i\vert_{\mathcal{R}(E(\Omega))_i})]\subseteq[\overline{\Omega\cap\mathbb{C}_i}]$, that is $\sigma_S(T\vert_{\mathcal{R}(E(\Omega))})\subseteq\overline{\Omega}$.
\end{proof}
\begin{corollary}[Fuglede Theorem]
Let $T=\int_{\sigma_S(T)}\lambda dE_\lambda$ be the spectral decomposition of a normal operator $T\in\mathcal{B}_R(\mathcal{H})$. If $S\in\mathcal{B}_R(\mathcal{H})$ commutes with $T$, then $S$ commutes with $E(\Omega)$ for every $\Omega\in\mathcal{A}_{\sigma_S(T)}^s$.
\end{corollary}
\begin{proof}
If $S\in\mathcal{B}_R(\mathcal{H})$ commutes with $T$, then $S$ commutes with $T_i$, and by \cite[Theorem 3.17]{6C}, $S$ commutes with $E^i(\Omega\cap\mathbb{C}_i)$ for all $\Omega\in\mathcal{A}_{\sigma_S(T)}^s$. Thus $S$ commutes with $E(\Omega)$ for all $\Omega\in\mathcal{A}_{\sigma_S(T)}^s$.
\end{proof}
\begin{corollary}[Fuglede-Putnam Theorem]
Suppose $T_1,T_2\in\mathcal{B}_R(\mathcal{H})$ are normal operators. If there exists $S\in\mathcal{B}_R(\mathcal{H})$ such that $T_1S=ST_2$, then $T_1^*S=ST_2^*$.
\end{corollary}
\begin{proof}
By Lemma \ref{6l4}, $(T_1)_i$ and $(T_2)_i$ are normal. Since $(T_1)_iS_i=S_i(T_2)_i$, $(T_1^*)_iS_i=S_i(T_2^*)_i$ (see,  \cite[Corollary 3.20]{6C}). It follows that $T_1^*S=ST_2^*$.
\end{proof}

\end{document}